\documentclass[12pt]{article}
\usepackage{amssymb,xcolor,epic}
\usepackage{amsmath}
\usepackage{graphicx}
\usepackage{enumerate}
\usepackage{mathrsfs}
\usepackage{latexsym}
\usepackage{psfrag}
\usepackage{mathrsfs}
\usepackage{multicol}
\usepackage{algorithm}
\usepackage{algpseudocode}

\newcommand{\qed}{\hfill $\square$}
\newcommand{\floor}[1]{\lfloor{ #1}\rfloor}

\newcommand{\nchn}{\binom{n}{\floor{\frac{n}{2}}}}

\newcommand{\lanp}{\La(n,P)}

\newcommand{\La}{{\mathrm La}}

\newcommand{\Remark}{\noindent{\bfseries Remark.} }

\newenvironment{proof}{\noindent{\em Proof}.}{\qed\bigskip}

\newtheorem{theorem}{Theorem}[section]

\newtheorem{proposition}[theorem]{Proposition}
\newtheorem{corollary}[theorem]{Corollary}
\newtheorem{conjecture}{Conjecture}[section]

\newcommand{\comments}[1]{}
\begin{document}

\title{Families of Subsets Without a Given Poset in the Interval Chains}

\author{
Jun-Yi Guo
\thanks{ Department of Mathematics, National Taiwan Normal University, Taipei 11677, Taiwan
{\tt Email:	davidguo@ntnu.edu.tw} supported by MOST-104-2115-M-003-010.}
\and
Fei-Huang Chang
\thanks{Division of Preparatory Programs for Overseas Chinese Students, National Taiwan Normal University, New Taipei 24449, Taiwan
{\tt Email:cfh@ntnu.edu.tw} supported by MOST-104-2115-M-003-008-MY2.}
\and
Hong-Bin Chen
\thanks{The Mathematical Society of the Republic of China, Taipei 10617, Taiwan 
{\tt Email:andanchen@gmail.com}}
\and
Wei-Tian Li
\thanks{Department of Applied Mathematics, National Chung Hsing University, Taichung 40227, Taiwan
{\tt Email:weitianli@nchu.edu.tw} supported by MOST-103-2115-M-005-003-MY2.}
}

\date{\small \today}

\maketitle

\begin{abstract}
For two posets $P$ and $Q$, we say $Q$ is $P$-free if there does not exist any order-preserving injection 
from $P$ to $Q$.
The speical case for $Q$ being the Boolean lattice $B_n$ is well-studied, 
and the optiamal value  is denoted as $\lanp$.
Let us define $\La(Q,P)$ to be the largest size of any $P$-free subposet of $Q$. 

In this paper, we give an upper bound for $\La(Q,P)$ when $Q$ is a double chain and $P$ is any graded poset,
which is better than the previous known upper bound, by means of finding the indpendence number of an auxiliary graph related to $P$. 
For the auxiliary graph, we can find its independence number in polynomial time. 
In addition, we give methods to construct the posets satisfying the Griggs-Lu conjecture.

\end{abstract}


\section{Background and main results}

In 1928, Sperner~\cite{Spe} determined the maximum size of an inclusion-free family (antichain) of subsets of $[n]:=\{1,2,\ldots,n\}$,  which is $\nchn$. Erd\H{o}s~\cite{Erd} generalized this result to the maximum size of a family without any $k$ mutually inclusive subsets for any given integer $k$. 
A poset $P=(P,\le_P)$ is a {\em weak subposet} of $Q=(Q,\le_Q)$ if there is an order-preserving injection $f$ 
from $P$ to $Q$ ($a\le_P b$ implies $f(a)\le_Q f(b)$), and $Q$ is $P$-free if there does not exist such an order-preserving injection. In the following, we use the term subposet instead of weak subposet for convenience. 
We say a poset is {\em connected} if its Hasse diagram is connected. 
All posets in the paper are finite and connected.
Using the language of poset theory, Sperner and Erd\H{o}s determined the largest size of a $P_k$-free subposet of 
the Boolean lattice $B_n=(2^{[n]},\subseteq)$, where $P_k$ is a chain on $k$ elements. 
For a poset $P$, the function $\lanp$ was first introduced in~\cite{DebKat}, defined to be the largest size of a $P$-free subposet of the Boolean lattice $B_n=(2^{[n]},\subseteq)$. 
The exact or asymptotic values of $\lanp$ for some specific posets have been studied, such 
as the $V$ poset~\cite{KatTar}, the forks~\cite{Tha,DebKat}, the butterfly~\cite{DebKatSwa}, the $N$ poset~\cite{GriKat}, the crowns~\cite{GriLu}, the generalized diamonds $D_k$~\cite{GriLiLu}, the harps~\cite{GriLiLu}, the complete 3-level posets $K_{r,s,t}$~\cite{Pat2}, and tree posets whose Hasse diagrams are cycle-free~\cite{Buk}. 
A {\em level} in a Boolean lattice is a collection of all subsets of the same size.
Let $e(P)$ be the maximum number $m$ such that the union of any $m$ consecutive levels in a Boolean lattice 
does not containing $P$ as a subposet.
The definition of $e(P)$ gives $\lanp \ge \sum_{i=0}^{m-1}\binom{n}{\lfloor(n-m+1)/2\rfloor+i}=(m+o_n(1))\nchn$. 
The most challenging problem in this area is the next conjecture.

\begin{conjecture}{\rm ~\cite{GriLu}}\label{Conj:GriLu}
For any poset $P$, the limit
\[\pi(P):=\lim_{n\rightarrow\infty}\frac{\lanp}{\nchn}\]
exists and is equal to $e(P)$.
\end{conjecture}
This conjecture was first posed by Griggs and Lu as a consequence of thier study of a vriety of posets and the early results of Katona and his collaborators. 
Although the known posets satisfying the above conjecture are just a small portion of all posets, 
we have no exceptions. Moreover, even for a small poset such as $D_2$, we do not know the limit yet. 
An upper bound for $\La(n,D_2)$ is obtained by Krammer et al.~\cite{KraMarYou} who use the powerful tool, flag algebras.
The best upper bound for $\La(n,D_2)$ is given in~\cite{GroMetTom2} recently.

Beside the value of $\lanp$ for a fixed $P$, some researchers study the upper bounds of $\lanp$ for general 
posets $P$ using the parameters of $P$. Since every poset on $k$ elements is a subposet of $P_k$,  
Erd\H{o}s's result on $P_{k}$-free families gives a natural upper bound for a poset $P$, 
namely, $\lanp \le (|P|-1)\nchn$.
The height of a poset $P$, denoted $h(P)$ (or just $h$ when the poset $P$ is specified), 
is the largest size of any chain 
in $P$. For any tree poset $T$, Bukh~\cite{Buk} proved that $\La(n,T)\sim (h(T)-1)\nchn$.   
However, it was pointed out, independently, by Jiang and by Lu~\cite{GriLu} that the height $h$ itself is not sufficient to bound $\lanp$.
 It turns out that neither the height nor the size can totally dominate $\lanp$. Later, 
Burcsi and Nagy~\cite{BurNag} gave the following bound using both $h$ and $|P|$.

\begin{theorem}{\rm\cite{BurNag}}~\label{BurNag} For any poset $P$, 
\begin{equation}\label{eq:BurNag}
\La(n,P)\le \left(\frac{|P|+h-2}{2}\right)\nchn.
\end{equation}
\end{theorem}

Chen and Li~\cite{ChenLi}, and Grosz et al.~\cite{GroMetTom} further independently improved it.

\begin{theorem}{\rm\cite{ChenLi}}~\label{ChenLi} Given a poset $P$
and $k\ge 1$, for sufficiently large $n$,
\[
\La(n,P)\le \frac{1}{k+1}\left(|P|+\frac{1}{2}(k^2+3k-2)(h-1)-1\right)\nchn.
\]
\end{theorem}

\begin{theorem}{\rm\cite{GroMetTom}}~\label{GroMetTom}
Given a poset $P$, for any integer $k\ge 2$, it holds that
\[
\La(n,P)\le \frac{1}{2^{k-1}}\left(|P|+(3k-5)(h2^{k-2}-1)-1\right)\nchn.
\]
\end{theorem}

The ideas in the proofs of above theorems are all similar to a double counting skill of Lubell~\cite{Lub}, 
which was used to prove the Sperner's theorem on antichains.
Their proofs consists of two steps: 
First, construct a class of isomorphic copies of some structural subposet $Q$ of $B_n$. 
Then determine an upper bound for the size of a $P$-free subposet of $Q$. 
Hence it is natrual to ask the question:
\begin{quotation}
What is the largest size of $P$-free subposet of a given poset $Q$? 
\end{quotation}
Let us denote the answer of this question by $\La(Q,P)$. The function $\lanp$ in the previous paragraphs is 
the case of $Q=B_n$. 
For other posets $Q$, Shahriari et al.~\cite{SarSha} studied $\La(L_n(q),B_2)$, where $L_n(q)$ is the subsapce lattice of an $n$-dimensional vector space over the field $\mathbb{F}(q)$. 
In~\cite{BurNag}, Burcsi and Nagy showed $\La(C_2,P)\le |P|+h-2$, where $C_2$ is called a {\em double chain}, and derived inequality~(\ref{eq:BurNag}). Chen and Li~\cite{ChenLi}, and Grosz et al.~\cite{GroMetTom} studied $\La(Q,P)$ for other posets $Q$ which called the {\em linkages} and the {\em interval chains}, and obtained Theorem~\ref{ChenLi} and~\ref{GroMetTom}, respectively.

We focus our study on $\La(C_2,P)$ for {\em graded posets} $P$.
A poset $P$ is graded if the number of elements in every maximal chain in $P$ is equal to $h$, 
the height of $P$. 
The levels, $L_i$, of a graded poset are defined inductively by letting $L_1$ be the set of all minimal elements of $P$, 
and $L_i$ be the set of all minimal elements of $P\setminus(\cup_{j=1}^{i-1} L_j)$ for $i\ge 2$.
Many posets studied earlier in the literature are graded.
The main result in our paper is that when $P$ is a graded poset, 
we have a strategy to construct a ``tight'' injection from $P$ to the double chain $C_2$, 
which enables us to reduce the bound~(\ref{eq:BurNag}) in Theorem~\ref{BurNag}.  
To state our main theorem, we have to construct an auxiliary graph $G_P$ for the graded posets. 
The precise definition of $G_P$ will be given in the next section. Now suppose we already have $G_P$. 
Let $\alpha(G_P)$ be the independence number of $G_P$. Then the following bound holds for $\La(C_2,P)$.

\begin{theorem}\label{main}
Let $P$ be a graded poset. Then 
\[\La(C_2,P)\le |P|+h-\alpha(G_P)-2.\]
\end{theorem}
The following corollary can be immediately deduced from the new upper bound for $\La(C_2,P)$ using the double counting method.

\begin{corollary}\label{maincor}
Let $P$ be a graded poset. We have 
\[
\La(n,P)\le \left(\frac{|P|+h-\alpha(G_P)-2}{2}\right)\nchn.
\]
\end{corollary}

The remaining sections of the paper are organized as follows. 
In Section 2, 
we give more connections between our results and the original $\lanp$ problem,  
and all the necessary terms for proving our main theorem. 
The proof of Theorem~\ref{main} and other related results are presented in Section 3. 
We give  a polynomial time algorithm for finding $\alpha(G_P)$, 
a method to construct more posets satisfying Conjecture~\ref{Conj:GriLu}, 
and some concluding remarks in the last section.

\section{Interval chains and the auxiliary graph}

We briefly introduce the work of Burcsi and Nagy~\cite{BurNag}, 
and of Grosz et al.~\cite{GroMetTom}. 
By a {\em full chain} $C$ in $B_n$, we refer to a collection of $n+1$ mutually inclusive subsets of $[n]$. Namely, $\emptyset\subset\{a_1\}\subset\{a_1,a_2\}\subset\cdots\subset [n]$. 
A {\em $k$-interval chain\/} $C_k$, named by Grosz et al., is the union of a full chain $\emptyset \subset
\{i_1\} \subset \{i_1,i_2\} \subset \cdots \subset [n]$
and the collection of sets $S$ satisfying
$\{i_1,\ldots,i_{m}\}\subset S\subset \{i_1,\ldots,i_{m+k}\}$ for $0\le m\le n-k$.
When $k=1$, a $k$-interval chain is merely a full chain $C$. 
For $k=2$, $C_2$ is just the double chain of Burcsi and Nagy~\cite{BurNag}. 
Observe that $C_k$ is a subposet of $C_{k+1}$ for each $k$, and $C_{n}$ is the Boolean lattice $B_n$. 
Thus, we have 
\begin{equation}\label{eq:IC}
\La(C_1,P)\le\La(C_2,P)\le \cdots \le \La(C_n,P)=\lanp. 
\end{equation}
Although the rightmost term in~(\ref{eq:IC}) is widely open for general posets, the leftmost term is 
always equal to $|P|-1$ for $n\ge |P|-1$. We are interested in determining other terms $\La(C_k,P)$. 
A technical proof of the next proposition on the upper bound for $\La(C_k,P)$ was given in~\cite{GroMetTom}. 
The special case $k=2$ was established earlier in~\cite{BurNag}. 

\begin{figure}[ht]
\begin{center}
\begin{picture}(100,90)
\put(20,0){\circle*{4}}
\put(20,20){\circle*{4}}
\put(20,40){\circle*{4}}
\put(20,60){\circle*{4}}
\put(20,80){\circle*{4}}
\put(20,0){\line(0,1){80}}
\put(20,0){\line(1,1){20}}
\put(20,20){\line(1,1){20}}
\put(20,40){\line(1,1){20}}
\put(20,60){\line(1,1){20}}
\put(40,20){\line(-1,1){20}}
\put(40,40){\line(-1,1){20}}
\put(40,60){\line(-1,1){20}}
\put(40,20){\circle*{4}$r_1$}
\put(40,40){\circle*{4}$r_2$}
\put(40,60){\circle*{4}$r_3$}
\put(40,80){\circle*{4}$r_4$}
\put(0,0){$\ell_0$}
\put(0,20){$\ell_1$}
\put(0,40){$\ell_2$}
\put(0,60){$\ell_3$}
\put(0,80){$\ell_4$}
\put(30,80){$\vdots$}
\end{picture}
\begin{picture}(100,90)
\put(20,0){\circle*{4}}
\put(20,20){\circle*{4}}
\put(20,40){\circle*{4}}
\put(20,60){\circle*{4}}
\put(20,80){\circle*{4}}
\put(20,0){\line(0,1){80}}
\put(20,0){\line(1,1){20}}
\put(20,0){\line(3,1){60}}
\put(20,20){\line(1,1){20}}
\put(40,20){\line(1,1){20}}
\put(20,20){\line(3,1){60}}
\put(20,40){\line(1,1){20}}
\put(40,40){\line(1,1){20}}
\put(20,40){\line(3,1){60}}
\put(20,60){\line(1,1){20}}
\put(40,60){\line(1,1){20}}
\put(20,60){\line(3,1){60}}
\put(40,20){\line(-1,1){20}}
\put(80,20){\line(-1,1){20}}
\put(80,20){\line(-2,1){40}}
\put(40,40){\line(-1,1){20}}
\put(60,40){\line(-2,1){40}}
\put(80,40){\line(-1,1){20}}
\put(80,40){\line(-2,1){40}}
\put(40,60){\line(-1,1){20}}
\put(60,60){\line(-2,1){40}}
\put(80,60){\line(-1,1){20}}
\put(80,60){\line(-2,1){40}}
\put(40,20){\circle*{4}}
\put(40,40){\circle*{4}}
\put(40,60){\circle*{4}}
\put(40,80){\circle*{4}}
\put(60,40){\circle*{4}}
\put(60,60){\circle*{4}}
\put(60,80){\circle*{4}}
\put(80,20){\circle*{4}}
\put(80,40){\circle*{4}}
\put(80,60){\circle*{4}}
\put(80,80){\circle*{4}}
\put(50,80){$\vdots$}
\end{picture}
\end{center}
\caption{A double chain and a 3-interval chain.}
\end{figure}
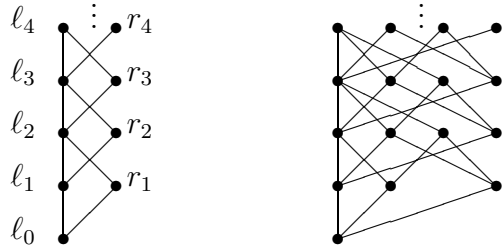

\begin{proposition}{\rm\cite{GroMetTom}}~\label{prop:GroMetTom}
Let $k\ge 2$. For any poset $P$ of size $|P|$ and height $h$, we have 
\[\La(C_k,P)\le |P|+(h-1)(3k-5)2^{k-2}-1.\]
\end{proposition}

To make the paper self-contained, we sketch the proof of Burcsi and Nagy for the case $k=2$ of Proposition~\ref{prop:GroMetTom}.
Let us denote the elements of the double chain $C_2$ by $\ell_0,\ell_1,\ldots, \ell_n$, and $r_1,r_2,\ldots,r_{n-1}$, 
where $\ell_0=\varnothing$, $\ell_{k}=\{i_1,i_2,\ldots,i_k\}$ for $1\le k\le n$, and $r_k=\ell_{k-1}\cup(\ell_{k+1}\setminus \ell_{k})$  for $1\le k\le n-1$. 
Arrange all elements in $C_2$ in the order:
\[\ell_0, \ell_1,r_1,\ell_2, r_2,\ldots, \ell_{n-1},r_{n-1},\ell_n.\]
Let $P$ be given. Consider a subset $F\subset C_2$ containing at least $|P|+h-1$ elements. 
Let us construct an order-preserving injection $f$ from $P$ to $F$. 
First apply Mirsky's theorem~\cite{Mir} to partition $P$ into $h$ antichains $A_1,\ldots, A_{h}$ 
so that for $a\in A_i$ and $a'\in A_j$, if $a\le a'$, then $i\le j$. 
For all elements in $A_1$, arbitrarily match them to the first $|A_1|$ elements in $F$. 
If we arbitrarily match elements in $A_2$ to the next $|A_2|$ elements in $F$, 
it may not hold all possible partial order relations among the elements of $A_1$ and $A_2$. 
When the $|A_1|$-th element in $F$ is $r_{i-1}$ or $\ell_i$, and $r_i\in F$, 
we avoid matching any element in $A_2$ to $r_i$ so as to preserve all possible partial order relations.
Hence we skip an element in $F$ when matching the elements in $A_1$ and $A_2$ to $F$.
Since $P$ is decomposed into $h$ antichains, 
we need to skip $h-1$ elements in $F$ when constructing the injection. 
Therefore, such an injection $f$ from $P$ to $F$ exists if $|F|\ge |P|+h-1$.
Consequently, a $P$-free subset in $C_2$ has size at most $|P|+h-2$.

Our aim is to reduce the number of skipped elements in the process of constructing the injection.
To achieve the goal, we will properly arrange the elements of $P$.
Such arrangements can be obtained by means of an auxiliary graph $G_P$ of the poset $P$.
Let $P$ be a graded poset with levels $L_1,L_2,\ldots,L_{h}$. 
We define the auxiliary graph $G_P$ with the vertex set consisting of the triples of mutually incomparable elements $\{x,y,z\}\subseteq P$ of two types:
A vertex of $G_P$ is {\em $V$-type} if $\{x,y,z\}$ satisfies $x\in L_i$ and  $y,z\in L_{i+1}$, and is {\em $\Lambda$-type} if $x,y\in L_i$ and $z\in L_{i+1}$.  Two vertices $v=\{x,y,z\}$ and $ v'=\{x',y',z'\}$
are adjacent in $G_P$ if and only if $v,v'\subset (L_i\cup L_{i+1})$ for some $i$, or  
$v\subset (L_{i-1}\cup L_i)$ and $v'\subset (L_i\cup L_{i+1})$ with one of the following conditions holding:
\begin{description}
\item{1.} $|L_i|\ge 5$ and $v\cap v'\neq \varnothing$.
\item{2.} $|L_i|=3$ or $4$, $v\cap v'\neq \varnothing$, 
except $v$ is $V$-type, $v'$ is $\Lambda$-type, and $|v\cap v'|=1$. 
\item{3.} $|L_i|=3$ or $4$, $v\cap v'=\varnothing$, $|v\cap L_i|+|v'\cap L_i|=|L_i|$. 
\item{4.} $|L_i|=2$.
 \end{description}

In Figure~\ref{Aux}, we present an example of $G_p$ of a graded poset $P$.
For this poset,  we have $\alpha(G_P)=3$.
Hence, we have $\La(C_2,P)\le |P|+(h-2)-\alpha(G_P)=16$ by Theorem~\ref{main}.

\begin{figure}[ht]
\begin{multicols}{3}

\begin{center}
\begin{picture}(100,180)
\put(0,160){$P$}
\put(15,20){\circle*{4}$z_1$}
\put(75,20){\circle*{4}$z_2$}
\put(0,50){\circle*{4}$y_1$}
\put(15,50){\circle*{4}$y_2$}
\put(35,50){\circle*{4}$y_3$}
\put(60,50){\circle*{4}$y_4$}
\put(90,50){\circle*{4}$y_5$}
\put(15,80){\circle*{4}$x_1$}
\put(75,80){\circle*{4}$x_2$}
\put(0,110){\circle*{4}$s_1$}
\put(30,110){\circle*{4}$s_2$}
\put(60,110){\circle*{4}$s_3$}
\put(90,110){\circle*{4}$s_4$}
\put(15,140){\circle*{4}$r_1$}
\put(45,140){\circle*{4}$r_2$}
\put(75,140){\circle*{4}$r_3$}
\put(0,50){\line(1,2){45}}  
\put(0,110){\line(1,2){15}}  
\put(15,20){\line(2,3){20}}  
\put(15,80){\line(2,-3){20}}  
\put(60,50){\line(1,2){30}}  
\put(60,110){\line(1,2){15}}  
\put(75,20){\line(1,2){15}}  
\put(0,50){\line(1,-2){15}}  
\put(0,110){\line(1,-2){15}}  
\put(15,140){\line(1,-2){15}}  
\put(45,140){\line(1,-2){45}}  
\put(60,50){\line(1,-2){15}}  
\put(75,140){\line(1,-2){15}}  
\put(15,20){\line(0,1){60}}  
\put(35,50){\line(4,-3){40}}  
\end{picture}
\end{center}

\columnbreak

\begin{center}
\begin{tabular}{|c|c|c|}
\hline vertex & triple & type  \\
\hline $t_1$ & $\{r_1,s_3,s_4\}$ & $\Lambda$\\
\hline $t_2$ & $\{r_2,s_1,s_4\}$ & $\Lambda$\\
\hline $t_3$ & $\{r_3,s_1,s_2\}$ & $\Lambda$ \\
\hline $t_4$ & $\{r_2,r_3,s_1\}$ & $V$ \\
\hline $t_5$ & $\{r_2,r_3,s_4\}$ & $V$ \\
\hline $u_1$ & $\{s_1,s_2,x_2\}$ & $V$\\
\hline $u_2$ & $\{s_3,s_4,x_1\}$ & $V$\\
\hline $v_1$ & $\{x_1,y_4,y_5\}$ & $\Lambda$\\
\hline $v_2$ & $\{x_2,y_1,y_2\}$ & $\Lambda$ \\
\hline $v_3$ & $\{x_2,y_1,y_3\}$ & $\Lambda$ \\
\hline $v_4$ & $\{x_2,y_2,y_3\}$ & $\Lambda$ \\ 
\hline $w_1$ & $\{y_1,y_2,z_2\}$ & $V$ \\ 
\hline $w_2$ & $\{y_4,y_5,z_1\}$ & $V$ \\ 
\hline   
\end{tabular}
\end{center}

\columnbreak

\begin{center}
\begin{picture}(120,190)
\put(0,170){$G_P$}
\put(30,15){\circle*{4}}
\put(90,15){\circle{6}}
\put(90,15){\circle*{4}}
\put(30,15){\line(1,0){60}}
\put(15,75){\circle*{4}}
\put(45,45){\circle*{4}}
\put(45,105){\circle*{4}}
\put(75,45){\circle*{4}}
\put(75,45){\circle{6}}
\put(75,105){\circle*{4}}
\put(105,75){\circle*{4}}
\put(15,75){\line(1,0){90}}
\put(45,45){\line(1,0){30}}
\put(45,105){\line(1,0){30}}
\put(15,75){\line(1,1){30}}
\put(15,75){\line(1,-1){30}}
\put(105,75){\line(-1,1){30}}
\put(105,75){\line(-1,-1){30}}
\put(15,75){\line(2,1){60}}
\put(15,75){\line(2,-1){60}}
\put(105,75){\line(-2,-1){60}}
\put(105,75){\line(-2,1){60}}
\put(45,45){\line(0,1){60}}
\put(75,45){\line(0,1){60}}
\put(45,45){\line(1,2){30}}
\put(75,45){\line(-1,2){30}}
\put(45,120){\circle*{4}}
\put(75,120){\circle*{4}}
\put(75,120){\circle{6}}
\put(30,150){\circle*{4}}
\put(90,150){\circle*{4}}
\put(60,165){\circle*{4}}
\put(45,120){\line(1,0){30}}
\put(30,150){\line(1,0){60}}
\put(45,120){\line(1,3){15}}
\put(75,120){\line(-1,3){15}}
\put(30,150){\line(2,1){30}}
\put(30,150){\line(3,-2){45}}
\put(90,150){\line(-3,-2){45}}
\put(90,150){\line(-2,1){30}}
\put(30,150){\line(1,-2){15}}
\put(90,150){\line(-1,-2){15}}
\put(30,120){$t_4$}
\put(80,120){$t_5$}
\put(15,150){$t_1$}
\put(95,150){$t_3$}
\put(65,165){$t_2$}
\put(0,75){$v_1$}
\put(30,45){$v_3$}
\put(25,100){$u_1$}
\put(80,45){$v_4$}
\put(85,100){$u_2$}
\put(110,75){$v_2$}
\put(15,15){$w_1$}
\put(95,15){$w_2$}
\put(45,120){\line(0,-1){15}}
\put(75,120){\line(0,-1){15}}
\put(30,150){\line(1,-1){45}}
\put(90,150){\line(-1,-1){45}}
\qbezier(30,150)(20,110)(45,105)
\qbezier(90,150)(100,110)(75,105)
\put(30,15){\line(1,2){30}}
\qbezier(15,75)(30,15)(90,15)
\qbezier(105,75)(90,15)(30,15)
\end{picture}
\end{center}

\end{multicols}
\caption{A graded poset and its auxilary graph. }\label{Aux}
\end{figure}
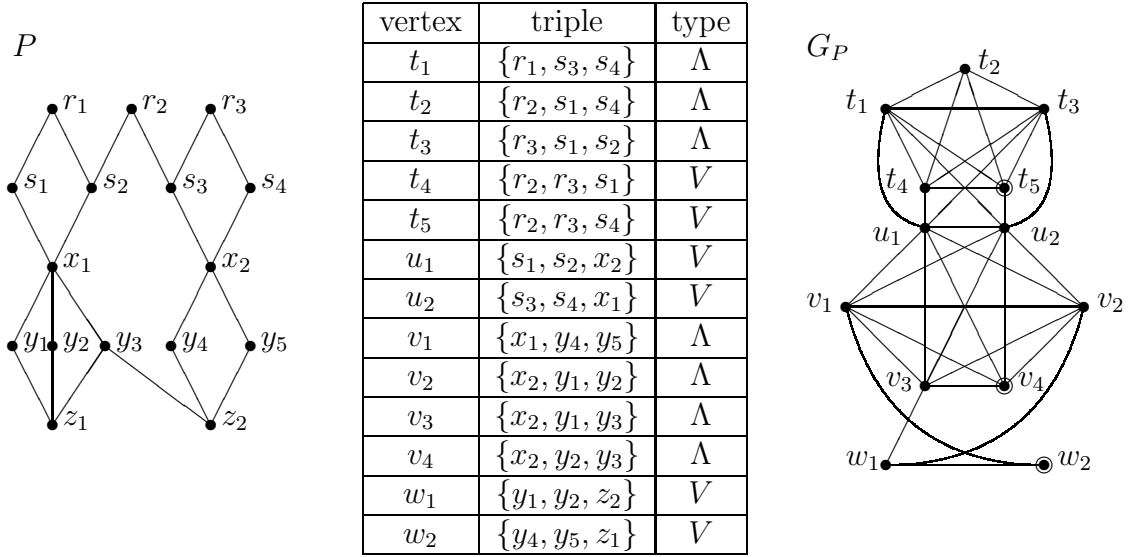

\section{Main Results}

In this section, we present our work. 
The first one is Theorem~\ref{main}. 
We will choose some specific triples of incomparable elements in a graded poset, 
and arrange them in a suitable way to avoid skipping elements when the construction goes up from one level to the next level. One triple is chosen, one element is saved.
Given any independent set of $G_P$, 
we choose the triples corresponding to the vertices of the independent set. 
Thus, the statement of Theorem~\ref{main} follows. Now we show the proof.

\subsection{Proof of Theorem~\ref{main}}  

Given a graded poset $P$ with height $h$ and levels $L_1,\ldots,L_{h}$, 
pick a subset $F$ of $C_2$ with size $|F|=|P|+h-\alpha(G_P)-1$.  
Let us construct an order-preserving injection $f$ from $P$ to $F$. 
In the sequel, we always arrange the elements in $F$ consistent with the order: 
$\ell_0, \ell_1,r_1,\ell_2, r_2,\ldots, \ell_{n-1},r_{n-1},\ell_n$.
Let $k=\alpha(G_P)$, and $I=\{v_1,v_2,\ldots, v_k\}$ be a maximum independent set of $G_P$. 
Because in the union of two consecutive levels $L_i\cup L_{i+1}$, 
$I$ contains at most one vertex corresponding to a triple in $L_i\cup L_{i+1}$, 
we may assume $v_1\subset (L_{n_1}\cup L_{n_1+1}), v_2\subset (L_{n_2}\cup L_{n_2+1}),\ldots, 
v_k\subset (L_{n_k}\cup L_{n_k+1})$ with $n_1<n_2<\cdots <n_k$.
To construct the desired injection $f$ from $P$ to $F$, we match elements in 
$L_1,  L_2,\ldots, L_h$,  to elements in $F$ consecutively.
Our basic rules are the following: 
For elements in the same level but not in any $v_j$, we match them consecutively to $F$. 
For elements in some $v_j$, 
we will last match the elements in $v_j\cap L_{n_j}$ to $F$ after all other elements in $L_{n_j}$ are matched, while the elements in $v_j\cap L_{n_j+1}$ will be first matched before all other elements in $L_{n_j+1}$ can be matched.
Now let us see how to avoid skipping elements when matching elements in $L_{n_j}\cup L_{n_j+1}$ to $F$. 

\noindent{\bf Case 1.} $v_j\cap v_{j+1}=\varnothing$.

Recall that there are two types of vertices in $G_P$.

\noindent{\bf Subcase 1.1} The vertex $v_j$ is of $\Lambda$-type. 

Suppose $v_j=\{x,y,z\}$ and $x,y\in L_{n_j}$. 
We match the elements $x$, $y$, and $z$ to $F$ consecutively, 
and do modification only when $f(z)=\ell_i$ and also $r_i$ is in $F$ for some $i$.   
The reason is that we might have matched $y$ to $r_{i-1}$  
and will have to match the next element $w$ in $L_{n_j+1}$ to $r_i$, but $w\ge_P y$.  
Thus, the injection fails to preserve the partial order relation between $y$ and $w$. 
If the above situation happens, we exchange $f(z)$ and $f(w)$ so that $f(z)=r_i$ and $f(w)=\ell_i$.
Thus, we do not skip the element $r_i$ and all possible partial order relations are preserved. 

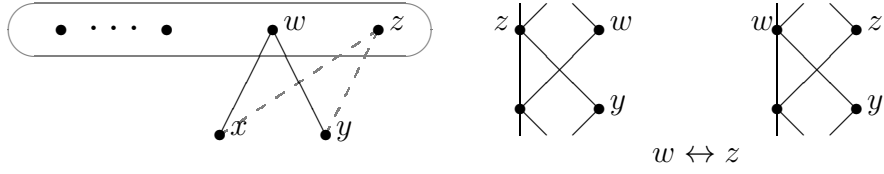
\begin{figure}[ht]
\begin{center}
\begin{picture}(160,70)
\textcolor{gray}{
\put(80,50){\oval(160,20)}
\dashline{4}(80,10)(140,50)
\dashline{4}(140,50)(120,10)
}
\put(20,50){\circle*{4}}	
\put(60,50){\circle*{4}}	
\put(100,50){\circle*{4}$w$}	
\put(140,50){\circle*{4}$z$}	
\put(120,10){\circle*{4}$y$}	
\put(80,10){\circle*{4}$x$}	
\put(80,10){\line(1,2){20}}	
\put(100,50){\line(1,-2){20}}	
\put(30,50){{\Large $\ldots$}}
\end{picture}\qquad 
\begin{picture}(60,70)
\put(10,20){\circle*{4}}	
\put(40,20){\circle*{4}$y$}	
\put(10,50){\circle*{4}}
\put(0,50){$z$}	
\put(40,50){\circle*{4}$w$}	
\put(10,10){\line(0,1){50}}
\put(10,20){\line(1,1){30}}
\put(10,50){\line(1,1){10}}
\put(40,20){\line(-1,-1){10}}
\put(40,50){\line(-1,1){10}}
\put(10,20){\line(1,-1){10}}
\put(10,50){\line(1,-1){30}}
\end{picture}$w\leftrightarrow z$
\begin{picture}(60,70)
\put(10,20){\circle*{4}}	
\put(40,20){\circle*{4}$y$}	
\put(10,50){\circle*{4}}	
\put(0,50){$w$}	
\put(40,50){\circle*{4}$z$}	
\put(10,10){\line(0,1){50}}
\put(10,20){\line(1,1){30}}
\put(10,50){\line(1,1){10}}
\put(40,20){\line(-1,-1){10}}
\put(40,50){\line(-1,1){10}}
\put(10,20){\line(1,-1){10}}
\put(10,50){\line(1,-1){30}}
\end{picture}
\end{center}
\caption{Switch $w$ and $z$. (The dotted lines indicate the incomparable triple.)}
\end{figure}
\noindent{\bf Subcase 1.2.} The vertex $v_j$ is of $V$-type.

Suppose $v_j=\{x,y,z\}$ and $x\in L_{n_j}$. 
As before, we first adopt the basic rules to match the elements $x$, $y$, and $z$ to $F$.
We do modification only when $f(x)=\ell_i$, $f(y)=r_i$, $f(z)=\ell_{i+1}$, and $r_{i+1}$ is in $F$.
It is because we might have matched the previous element $w\in L_{n_j}$ to $r_{i-1}$, but $w\le_P y$.
Then the injection cannot preserve the partial order relation.  
To fix it, we change $f(x)$ to be $r_i$ and $f(y)$ to be $\ell_{i}$.
However, when we next match $z$ to $\ell_{i+1}$ and some element $u$ in $L_{n_j+1}$ to $r_{i+1}$, 
we may encounter the situation $u\ge_P x$. Then the partial order relation cannot be preserved again. 
Once this happens, we exchange $f(z)$ and $f(u)$ so that $f(z)=r_{i+1}$ and $f(u)=\ell_{i+1}$ to preserve the partial order relations. Thus, we do not skip any element. 

\begin{figure}[ht]
\begin{center}
\begin{picture}(160,70)
\textcolor{gray}{
\put(80,50){\oval(160,20)}
\dashline{4}(100,50)(120,10)
\dashline{4}(140,50)(120,10)
}
 \put(20,50){\circle*{4}}	
\put(60,50){\circle*{4}$u$}	
\put(100,50){\circle*{4}$z$}	
\put(140,50){\circle*{4}$y$}	
\put(120,10){\circle*{4}$x$}	
\put(80,10){\circle*{4}$w$}	
\put(80,10){\line(1,2){20}}	
\put(80,10){\line(3,2){60}}	
\put(60,50){\line(1,-2){20}}	
\put(60,50){\line(3,-2){60}}	
\put(30,50){{\Large $\ldots$}}
\end{picture} 
\qquad
\begin{picture}(60,70)
\put(10,20){\circle*{4}}	
\put(40,20){\circle*{4}$w$}	
\put(10,50){\circle*{4}}	
\put(40,50){\circle*{4}}	
\put(0,35){$x$}	
\put(10,35){\circle*{4}}	
\put(40,35){\circle*{4}$y$}	
\put(10,10){\line(0,1){50}}
\put(10,20){\line(2,1){30}}
\put(10,50){\line(1,1){10}}
\put(40,20){\line(-1,-1){10}}
\put(40,50){\line(-1,1){10}}
\put(10,20){\line(1,-1){10}}
\put(10,50){\line(2,-1){30}}
\put(10,20){\line(2,1){30}}
\put(10,35){\line(2,-1){30}}
\put(10,35){\line(2,1){30}}
\end{picture} $x\leftrightarrow y$ 
\begin{picture}(60,70)
\put(10,20){\circle*{4}}	
\put(40,20){\circle*{4}$w$}	
\put(10,50){\circle*{4}}	
\put(40,50){\circle*{4}$u$}	
\put(10,35){\circle*{4}}	
\put(0,35){$y$}	
\put(0,50){$z$}	
\put(40,35){\circle*{4}$x$}	
\put(10,10){\line(0,1){50}}
\put(10,20){\line(2,1){30}}
\put(10,50){\line(1,1){10}}
\put(40,20){\line(-1,-1){10}}
\put(40,50){\line(-1,1){10}}
\put(10,20){\line(1,-1){10}}
\put(10,50){\line(2,-1){30}}
\put(10,20){\line(2,1){30}}
\put(10,35){\line(2,-1){30}}
\put(10,35){\line(2,1){30}}
\end{picture}$u\leftrightarrow z$ 
\begin{picture}(60,70)
\put(10,20){\circle*{4}}	
\put(40,20){\circle*{4}$w$}	
\put(40,50){\circle*{4}$z$}	
\put(10,35){\circle*{4}}	
\put(40,35){\circle*{4}$x$}	
\put(10,50){\circle*{4}}	
\put(0,35){$y$}	
\put(0,50){$u$}	
\put(10,10){\line(0,1){50}}
\put(10,20){\line(2,1){30}}
\put(10,50){\line(1,1){10}}
\put(40,20){\line(-1,-1){10}}
\put(40,50){\line(-1,1){10}}
\put(10,20){\line(1,-1){10}}
\put(10,50){\line(2,-1){30}}
\put(10,20){\line(2,1){30}}
\put(10,35){\line(2,-1){30}}
\put(10,35){\line(2,1){30}}
\end{picture}
\end{center}
\caption{Switch $x$ and $y$, and also $u$ and $z$.}
\end{figure}
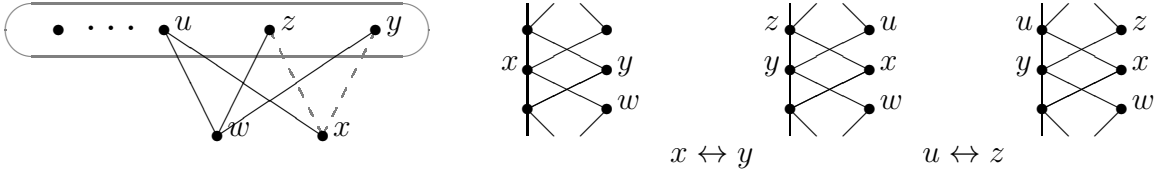

\noindent{\bf Case 2.}  $v_j\cap v_{j+1}\neq\varnothing$.

In this case, $v_j\subset (L_{n_j}\cup L_{n_j+1})$, $v_{j+1}\subset (L_{n_{j+1}}\cup L_{n_{j+1}+1})$, 
and $n_j+1=n_{j+1}$. 
By the adjacency relation of vertices in $G_P$, 
$v_j$ is of $V$-type, $v_{j+1}$ is of $\Lambda$-type, and $|L_{n_j+1}|=3$ or $4$.
Suppose that $v_j=\{x,y,z\}$, $v_{j+1}=\{x',y',z'\}$, and $z=x'$ is the common element in the two vertices.

\noindent{\bf Subcase 2.1.} $|L_{n_j+1}|=3$.

We match the elements in the order $x$, $y$, $z$, $y'$, and $z'$ to $F$.
If following the above order leads to $f(x)=\ell_i$ and $f(y)=r_i$, then we switch them so that $f(x)=r_i$ and $f(y)=\ell_i$. Thus, for any element $s\in L_{n_j}-\{x\}$, we have $f(y)\ge f(s)$. This switch preserves all possible relations between $y$ and the elements in $L_{n_j}-\{x\}$. 
For the next two elements $x'$ and $y'$, 
if we match them to $\ell_{i+1}$ and $r_{i+1}$, respectively, 
then we may loss the possible partial order relation between $x$ and $y'$ like the subcase 1.2.
Should this happen, we use the same way to solve it. 
Namely, switch $f(x')$ and $f(y')$ to get $f(x')=r_{i+1}$ and $f(y')=\ell_{i+1}$.  
Moreover, if $f(z')=\ell_{i'}$ and there is an element $w$ in the same level, which has to be matched to $r_{i'}$, then we switch so that $f(z')=r_{i'}$ and $f(w)=\ell_{i'}$ to preserve the possible relations between $w$ and $x'$ and $y'$
as  in Subcase 1.1. Therefore, we do not skip any element in $F$ when constructing the injection for the five elements.

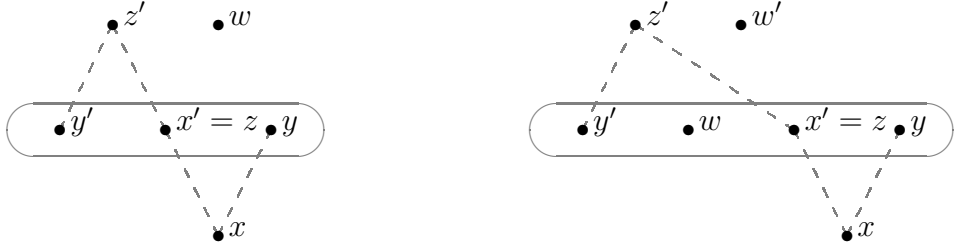
\begin{figure}[ht]
\begin{center}
\begin{picture}(140,100)
\textcolor{gray}{
\put(60,50){\oval(120,20)}
\dashline{4}(40,90)(20,50)
\dashline{4}(40,90)(60,50)
\dashline{4}(80,10)(100,50)
\dashline{4}(80,10)(60,50)}
\put(40,90){\circle*{4}$z'$}	
\put(80,90){\circle*{4}$w$}	
\put(20,50){\circle*{4}$y'$}	
\put(60,50){\circle*{4}$x'=z$}	
\put(100,50){\circle*{4}$y$}	
\put(80,10){\circle*{4}$x$}	
\end{picture}	
\hspace{50pt}
\begin{picture}(160,100)
\textcolor{gray}{
\put(80,50){\oval(160,20)}
\dashline{4}(40,90)(100,50)
\dashline{4}(40,90)(20,50)
\dashline{4}(100,50)(120,10)
\dashline{4}(140,50)(120,10)}
\put(40,90){\circle*{4}$z'$}	
\put(80,90){\circle*{4}$w'$}	
\put(20,50){\circle*{4}$y'$}	
\put(60,50){\circle*{4}$w$}	
\put(100,50){\circle*{4}$x'=z$}	
\put(140,50){\circle*{4}$y$}	
\put(120,10){\circle*{4}$x$}	
\end{picture}	
\end{center}
\caption{Intersecting triples but considered as nonadjacent vertices.}
\end{figure}

\noindent{\bf Subcase 2.2.} $|L_{n_j+1}|=4$.

Let $w$ be the element in $L_{n_j+1}$ but not in any triples.
We pairwisely match the elements in the order $x$ and $y$, $w$ and $z$, and $y'$ and $z'$ to $F$.
For $x$ and $y$,
if following the above order leads to $f(x)=\ell_i$ and $f(y)=r_i$, 
then we switch them so that $f(x)=r_i$ and $f(y)=\ell_i$; otherwise, we do not change.
The switch above preserves all possible relations between $y$ and the elements in $L_{n_j}-\{x\}$. 
Then for $z$ and $w$, if we have switched $x$ and $y$ in the previous step, and have to match 
$z$ and $w$ to $\ell_{i+1}$ and $r_{i+1}$, then we switch them so that $f(z)=r_{i+1}$ and $f(w)=\ell_{i+1}$;  
otherwise, we do not change.
The switch preserves the possible relation between $x$ and $w$. 
Finally, for $y'$ and $z'$, we first simply match $y'$ to the next element in $F$.
For $z'$, we will take the next element, say $w'$, in the same level as $z'$ together into consideration.
If we have to match $z'$ and $w'$ to some $\ell_{i'}$ and $r_{i'}$, respectively, 
then we switch them so that $f(z')=r_{i'}$ and $f(w')=\ell_{i'}$  as in the subcase 1.1. 
This switch preserves the possible relation between $w'$ and elements in $L_{n_j+1}$.

Since we need to skip elements only when the union of the two consecutive levels contains no triple selected in $I$, 
the number of skipped elements in $F$ is ($h-1)-\alpha(G_P)$. 
Thus, there is an order-preserving injection from $P$ to $F$. 
Consequently, $\La(D,P)\le |P|+h-\alpha(G_P)-2$.

\subsection{The adjacent vertices in $G_P$}  

In this section, 
we explain why we can only choose the triples corresponding to the vertices in an independent of $G_P$. 
Suppose we choose two triples corresponding a pair adjacent vertices in $G_P$. 
First, we may assume the incomparble triples of elements are all in $L_j\cup L_{j+1}$ for some $j$.
When using the method of Burcsi and Nagy to construct the injection, 
we skip only one element when the construction goes up from $L_j$ to $ L_{j+1}$. 
Therefore, it is impossible to save two elements no matter how we arrange them. 
So, we can only choose an incomparble triple of elements in $L_j\cup L_{j+1}$.
Now we assume that one of the triples is in $L_{j-1}\cup L_j$ and the other is in $L_j\cup L_{j+1}$ for some $j$.
We give examples showing that reducing two skipped elements is impossible under some circumstances. 
Consider the following posets $P_1,\ldots,P_7$ whose Hasse diagrams are presented in Figure~\ref{EDGES}:
\begin{description}
\item{$P_1$:} $\{x_1,x_2,y_1,\ldots y_n,z_1,z_2\}$ and $n\ge 3$; 
$x_1\le y_k\le z_1$ for $k\ge 3$, and $x_2\le y_k\le z_2$ for $k\ge 1$.
\item{$P_2$:} $\{x_1,x_2,y_1,\ldots y_n,z_1,z_2\}$ and $n\ge 5$;
$x_1\le y_k\le z_2$ for $k\ge 1$, $x_2\le y_k$ for $k\neq 2,3$, $y_k\le z_1$ for $k\neq 1,2$.
\item{$P_3$:} $\{x_1,x_2,x_3,y_1,\ldots y_n,z_1,z_2,z_3\}$ and $n\ge 2$;
$x_3\le y_1\le z_3$  and $x_m\le y_k\le z_m$ for all $m$ and $k\ge 2$.
\item{$P_4$:} $\{x_1,x_2,y_1,\ldots y_n,z_1,z_2,z_3\}$ and $n\ge 3$,
$x_1\le y_k$ for $k\ge 3$, $x_2\le y_3$ for all $k$, $ y_k\le z_1$ for all $k\ge 3$, and $ y_k\le z_2$ for all $k\ge 3$.
\item{$P_5$:} $\{x_1,x_2,y_1,y_2,y_3, y_4,z_1,z_2\}$;
$x_1\le y_k$ for all $k$, $x_2\le y_k$ for $k=1,2$, $y_k\le z_1$ for $k=3,4$, and $y_k\le z_2$ for all $k$.
\item{$P_6$:} $\{x_1,x_2,y_1,y_2, y_3,z_1,z_2,z_3\}$;
$x_1\le y_k$ for all $k$, $x_2\le y_1$, $y_k\le z_1$ for $k=2,3$, $y_k\le z_2$ for $k=2,3$,
 and $y_k\le z_3$ for all $k$.
\item{$P_7$:} $\{x_1,x_2,x_3,y_1, y_2,z_1,z_2,z_3\}$;
$x_k\le y_1\le z_3$ and $x_1\le y_2\le z_k$ for all $k$.
\end{description}

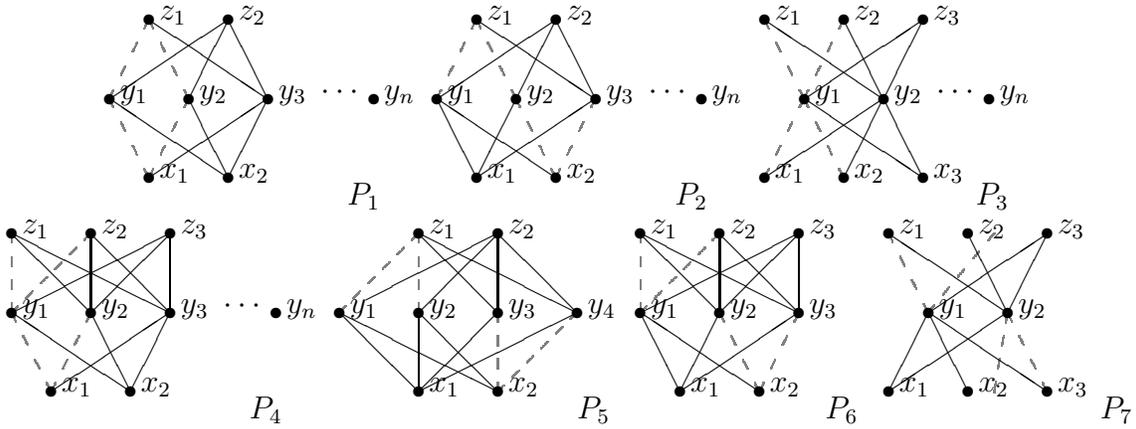
\begin{figure}[ht]
\begin{center}
\begin{picture}(120,80)
\textcolor{gray}{
\dashline{4}(25,70)(10,40)
\dashline{4}(25,70)(40,40)
\dashline{4}(25,10)(10,40)
\dashline{4}(25,10)(40,40)}
\put(25,70){\circle*{4}$z_1$}	
\put(55,70){\circle*{4}$z_2$}	
\put(25,10){\circle*{4}$x_1$}	
\put(90,40){$\cdots$}
\put(110,40){\circle*{4}$y_n$}
\put(10,40){\circle*{4}$y_1$}	
\put(40,40){\circle*{4}$y_2$}	
\put(70,40){\circle*{4}$y_3$}	
\put(55,10){\circle*{4}$x_2$}	
\put(55,70){\line(1,-2){15}}
\put(55,70){\line(-1,-2){15}}
\put(55,70){\line(-3,-2){45}}
\put(25,70){\line(3,-2){45}}
\put(55,10){\line(-3,2){45}}
\put(25,10){\line(3,2){45}}
\put(55,10){\line(1,2){15}}
\put(55,10){\line(-1,2){15}}
\put(100,0){$P_1$}
\end{picture}
\begin{picture}(120,80)
\textcolor{gray}{
\dashline{4}(25,70)(10,40)
\dashline{4}(25,70)(40,40)
\dashline{4}(55,10)(70,40)
\dashline{4}(55,10)(40,40)}
\put(25,70){\circle*{4}$z_1$}	
\put(55,70){\circle*{4}$z_2$}	
\put(25,10){\circle*{4}$x_1$}	
\put(90,40){$\cdots$}
\put(110,40){\circle*{4}$y_n$}
\put(10,40){\circle*{4}$y_1$}	
\put(40,40){\circle*{4}$y_2$}	
\put(70,40){\circle*{4}$y_3$}	
\put(55,10){\circle*{4}$x_2$}	
\put(55,70){\line(1,-2){15}}
\put(55,70){\line(-1,-2){15}}
\put(55,70){\line(-3,-2){45}}
\put(25,70){\line(3,-2){45}}
\put(55,10){\line(-3,2){45}}
\put(25,10){\line(3,2){45}}
\put(25,10){\line(1,2){15}}
\put(25,10){\line(-1,2){15}}
\put(100,0){$P_2$}
\end{picture}
\begin{picture}(110,80)
\textcolor{gray}{
\dashline{4}(25,40)(10,10)
\dashline{4}(25,40)(40,10) 
\dashline{4}(25,40)(10,70)
\dashline{4}(25,40)(40,70)}
\put(10,10){\circle*{4}$x_1$}	
\put(40,10){\circle*{4}$x_2$}	
\put(70,10){\circle*{4}$x_3$}	
\put(25,40){\circle*{4}$y_1$}	
\put(55,40){\circle*{4}$y_2$}	
\put(75,40){$\cdots$}
\put(95,40){\circle*{4}$y_n$}	
\put(10,70){\circle*{4}$z_1$}	
\put(40,70){\circle*{4}$z_2$}	
\put(70,70){\circle*{4}$z_3$}	
\put(55,40){\line(1,-2){15}}
\put(55,40){\line(-1,-2){15}}
\put(55,40){\line(-3,-2){45}}
\put(25,40){\line(3,-2){45}}
\put(55,40){\line(-3,2){45}}
\put(25,40){\line(3,2){45}}
\put(55,40){\line(1,2){15}}
\put(55,40){\line(-1,2){15}}
\put(90,0){$P_3$}
\end{picture}	
\\
\begin{picture}(120,80)
\textcolor{gray}{
\dashline{4}(40,70)(10,40)
\dashline{4}(10,70)(10,40)
\dashline{4}(25,10)(10,40)
\dashline{4}(25,10)(40,40)}
\put(10,70){\circle*{4}$z_1$}	
\put(40,70){\circle*{4}$z_2$}	
\put(70,70){\circle*{4}$z_3$}	
\put(25,10){\circle*{4}$x_1$}	
\put(10,40){\circle*{4}$y_1$}	
\put(40,40){\circle*{4}$y_2$}	
\put(90,40){$\cdots$}
\put(110,40){\circle*{4}$y_n$}
\put(70,40){\circle*{4}$y_3$}	
\put(55,10){\circle*{4}$x_2$}	
\put(55,10){\line(-3,2){45}}
\put(55,10){\line(1,2){15}}
\put(55,10){\line(-1,2){15}}
\put(25,10){\line(3,2){45}}
\put(70,70){\line(0,-1){30}}
\put(70,70){\line(-1,-1){30}}
\put(70,70){\line(-2,-1){60}}
\put(40,70){\line(0,-1){30}}
\put(40,70){\line(1,-1){30}}
\put(10,70){\line(2,-1){60}}
\put(10,70){\line(1,-1){30}}
\put(100,0){$P_4$}
\end{picture}
\begin{picture}(110,80)
\textcolor{gray}{
\dashline{4}(40,70)(40,40)
\dashline{4}(40,70)(10,40)
\dashline{4}(70,40)(70,10)
\dashline{4}(100,40)(70,10)}
\put(40,10){\circle*{4}$x_1$}	
\put(70,10){\circle*{4}$x_2$}	
\put(40,70){\circle*{4}$z_1$}	
\put(70,70){\circle*{4}$z_2$}	
\put(10,40){\circle*{4}$y_1$}	
\put(40,40){\circle*{4}$y_2$}	
\put(70,40){\circle*{4}$y_3$}	
\put(100,40){\circle*{4}$y_4$}	
\put(40,70){\line(1,-1){30}}
\put(40,70){\line(2,-1){60}}
\put(70,70){\line(1,-1){30}}
\put(70,70){\line(-1,-1){30}} 
\put(70,70){\line(0,-1){30}}
\put(70,70){\line(-2,-1){50}}
\put(70,10){\line(-1,1){30}} 
\put(70,10){\line(-2,1){60}}
\put(40,10){\line(1,1){30}}
\put(40,10){\line(-1,1){30}} 
\put(40,10){\line(0,1){30}}
\put(40,10){\line(2,1){60}}
\put(100,0){$P_5$}
\end{picture}
\begin{picture}(90,80)
\textcolor{gray}{
\dashline{4}(40,70)(10,40)
\dashline{4}(10,70)(10,40)
\dashline{4}(55,10)(70,40)
\dashline{4}(55,10)(40,40)}
\put(10,70){\circle*{4}$z_1$}	
\put(40,70){\circle*{4}$z_2$}	
\put(70,70){\circle*{4}$z_3$}	
\put(25,10){\circle*{4}$x_1$}	
\put(10,40){\circle*{4}$y_1$}	
\put(40,40){\circle*{4}$y_2$}	
\put(70,40){\circle*{4}$y_3$}	
\put(55,10){\circle*{4}$x_2$}	
\put(55,10){\line(-3,2){45}}
\put(25,10){\line(1,2){15}}
\put(25,10){\line(-1,2){15}}
\put(25,10){\line(3,2){45}}
\put(70,70){\line(0,-1){30}}
\put(70,70){\line(-1,-1){30}}
\put(70,70){\line(-2,-1){60}}
\put(40,70){\line(0,-1){30}}
\put(40,70){\line(1,-1){30}}
\put(10,70){\line(2,-1){60}}
\put(10,70){\line(1,-1){30}}
\put(80,0){$P_6$}
\end{picture} 
\begin{picture}(100,80)
\textcolor{gray}{
\dashline{4}(55,40)(70,10)
\dashline{4}(55,40)(50,10) 
\dashline{4}(25,40)(10,70)
\dashline{4}(25,40)(50,70)}
\put(10,10){\circle*{4}$x_1$}	
\put(40,10){\circle*{4}$x_2$}	
\put(70,10){\circle*{4}$x_3$}	
\put(25,40){\circle*{4}$y_1$}	
\put(55,40){\circle*{4}$y_2$}	
\put(10,70){\circle*{4}$z_1$}	
\put(40,70){\circle*{4}$z_2$}	
\put(70,70){\circle*{4}$z_3$}	
\put(55,40){\line(-3,-2){45}}
\put(25,40){\line(3,-2){45}}
\put(55,40){\line(-3,2){45}}
\put(25,40){\line(3,2){45}}
\put(55,40){\line(1,2){15}}
\put(55,40){\line(-1,2){15}}
\put(25,40){\line(1,-2){15}}
\put(25,40){\line(-1,-2){15}}
\put(90,0){$P_7$}
\end{picture}	
\end{center}
\caption{Seven types of posets.}\label{EDGES}
\end{figure}

Each poset contains two triples of incomparable elements, 
one in the union of bottom and middle levels, 
and the other in in the union of top and middle levels. 
For every poset $P_i$, the auxiliary graph $G_{P_i}$ is $K_2$.
Indeed, these are all the possible configurations (up to the dual cases) of triples corresponding to the adjacent vertices in $G_P$.
By Theorem~\ref{main}, we have $\La(C_2,P_i)\le |P_i|+h(P_i)-\alpha(G_{P_i})-2=|P_i|$.  
We claim that the upper bound cannot be improved to $|P_i|-1$ by showing the next proposition.

\begin{proposition}
For each of the seven types of posets $P_i$ above, there is a $P_i$-free subset $F$ of $C_2$ with size $|P_i|$. 
\end{proposition}

\begin{proof}
For $P_1$ and $P_2$, we verify that $F_1=\{r_{i+1}, r_{i+2},\ldots, r_{i+n+2}\}\cup\{\ell_{i+2},\ell_{i+n+1}\}$ is both $P_1$-free andr $P_2$-free. In $P_1$, the element $x_2$ is less than all elements but $x_1$, 
so we have to match it to $r_{i+1}$. By symmetry, the element $z_2$ is matched to $r_{i+n+2}$.
Meanwhile, we have to match $x_1$ to $r_{i+2}$, and $z_1$ to $r_{i+n+1}$.
Each element $y_k$, $3\le k\le n$, is greater than two elements and also greater than two elements. 
Thus, we can only match them to the elements in $\{r_{i+3},\ldots, r_{i+n}\}$. 
However, $x_1\le y_k$ for all $k\ge 3$, but $r_{i+2}$ is not less than $r_{i+3}$. 
Hence, we cannot construct an order-preserving injection from $P_1$ to $F_1$.  
In $P_2$, the element $x_1$ is less than all elements but $x_2$, so we have to match it to $r_{i+1}$. 
By symmetry, the element $z_2$ is matched to $r_{i+n+2}$.
Meanwhile, we have to match $x_2$ to $r_{i+2}$, and $z_1$ to $r_{i+n+1}$.
Then each element $y_k$, $3\le k\le n$, is less than two elements and also greater than two elements. 
Thus, we can only match them to the elements in $\{r_{i+3},\ldots, r_{i+n}\}$. 
However, $x_1\le y_k$ for $k\ge 3$ but $r_{i+2}$ is not less than $r_{i+3}$. 
Hence, we cannot construct an order-preserving injection from $P_2$ to $F_1$.  

For $P_3$, consider the subset $F_2=\{r_{i+1}, r_{i+2},\ldots, r_{i+n+4}\}\cup\{\ell_{i+2},\ell_{i+n+3}\}$.
Each element $y_k$, $2\le k\le n$, is less than three elements and also greater than three elements. 
Thus, to preserve the partial order relation, 
we can only match them to the elements in $\{r_{i+4},\ldots, r_{i+n+1}\}$. 
Clearly, this matching is not possible.

The next is $P_4$. Let $F_3=\{r_{i+1}, r_{i+2},\ldots, r_{i+n+4}\}\cup\{\ell_{i+2}\}$ .
The element $x_2$ is less than all elements but $x_1$, so we have to match it to $r_{i+1}$. 
Moreover, we have to match $x_1$ to $r_{i+2}$.
Thus, $y_1$ and $y_2$ are matched to $\ell_{i+2}$ and $r_{i+3}$, 
because all other elements, $y_3,\ldots, y_n$ and $z_1,z_2,z_3$ are greater than $x_1$.
Note that $z_1,z_2,z_3$ must be matched to $r_{i+n+2}$, $r_{i+n+3}$, and $r_{i+n+4}$, 
since they are not less than other elements in $P_4$.
However, the element $r_{i+n+1}$ is incomparable with $r_{i+n+2}$, so we can only match 
$y_1$ to it. But this cannot be done since we have matched $y_1$ to another element in $F$.

For $P_5$ and $P_6$, pick $F_4=\{\ell_{i-1}, \ell_{i}, \ell_{i+1}\}\cup\{r_{i-2}, r_{i-1}, r_i, r_{i+1}, r_{i+2}\}$.
Observe that $x_1$ is less than all elements but $x_2$ in $P_5$, hence we can only match it to $r_{i-2}$. 
Similarly, $z_2$ has to be matched to $r_{i+2}$. This forces us to match $\{y_1,y_2,y_3,y_4\}$ to $\{\ell_{i-1}, \ell_{i}, r_i, \ell_{i+1},\}$, $x_2$ to $r_{i-1}$, and $z_1$ to $r_{i+1}$.
Furthermore, this again forces us to match $\{y_1,y_2\}$ to $\{r_i,\ell_{i+1}\}$, 
and also $\{y_3,y_4\}$ to $\{r_i,\ell_{i-1}\}$, which is impossible. So we conclude that $F$ is $P_5$-free.
The element $x_1$ in $P_6$ is less than all elements but $x_2$, hence we can only match it to $r_{i-2}$. 
Also, this forces that $x_2$ has to be matched to $r_{i-1}$.
Since $y_1$, $z_1$, $z_2$, $z_3$ are all greater than $x_2$, 
the elements $\ell_{i-1}$ and $r_i$ can only be matched by $y_2$ and $y_3$.
However, $r_i$ is less than only two elements in $F_4$, 
but both $y_2$ and $y_3$ are less than three elements in $P_6$. 
The partial order cannot be preserved.   
Hence $F_4$ is $P_6$-free.

Finally, we take $F_5=\{\ell_{i-2}, \ell_i, \ell_{i+2}\}\cup \{r_{i-2}, r_{i-1}, r_i, r_{i+1}, r_{i+2}\}$, 
and show it is $P_7$-free.
The element $y_1$ is less than one element and greater than three elements in $P_7$. 
Hence we can only match it to either $\ell_i$ or $r_{i+1}$. 
Similarly, the element $y_2$ can only be matched to $\ell_i$ or $r_{i-1}$.
So either $y_1$ is matched to $r_{i+1}$, or $y_2$ is matched to $r_{i-1}$.
By symmetry, we may assume that $y_1$ is matched to $r_{i+1}$. 
Then this forces $y_2$ being matched to $r_{i-1}$, 
since $\ell_i$ is only less than two elements except for $r_{i+1}$,
and $r_{i+1}$ is matched by $y_1$ already. 
However, we then have to match $\{z_1,z_2,z_3\}$ to $\{\ell_{i+2},r_{i+2},\ell_{i}\}$, 
and also $\{x_1,x_2,x_3\}$ to $\{\ell_{i-2},r_{i-2},\ell_{i}\}$. 
This is obviously impossible.
\end{proof}

Hence, we can reduce at most $\alpha(G_P)$ elements from the $h-1$ skipped ones when constructing an order preserving injection from a graded poset $P$ to a subset of $C_2$. 

\section{Concluding remarks} 

The quantity $\alpha(G_P)$ is not a well-studied parameter of posets. 
In addition, finding the independence number of a graph $G$ is a 
NP-complete problem in general~\cite{Coo,Kar}. 
Then, why do we derive such an unusual upper bound? 
In this section, we point two features of our upper bound. 
One is that finding the independence number of the auxiliary graph of the graded posets can be done in the polynomial time; 
the other is that we can use the incomparable triples to construct posets satisfying Conjecture~\ref{Conj:GriLu}. 

\subsection{Computing $\alpha(G_P)$ in polynomial time}

Given a subset $U$ of the vertex set of a graph, let $G[U]$ be the graph induced by vertices in $U$. 
We partition $V(G_P)$ into $V_1\cup V_2\cup\cdots\cup V_{h-1}$ 
such that $V_j$ is the set of vertices corresponding to the triples of incomparable elements in $L_j\cup L_{j+1}$.
Then each $G[V_j]$ is a clique in $G_P$ by definition. 
Hence, any independent set of $G_P$ contains at most one vertex in each $V_j$.
Moreover, the following proposition holds for $G_P$.

\begin{proposition}\label{ama} 
For any $1\le j\le h-1$, we have
\[
\alpha(G[\cup_{i=1}^j V_i])=\max_{v\in V_j}\{\alpha(G[(\cup_{i=1}^{j-1} V_i)\cup\{v\}])\}.
\]
\end{proposition}

\begin{proof}
The left-hand side is never smaller than the right-hand side in the equality, 
since $G[(\cup_{i=1}^{j-1} V_i)\cup\{v\}]$ is a subgraph of $G[\cup_{i=1}^j V_i]$ for any $v\in V_j$. 
To show the other direction, 
if a maximum independent set of $G[\cup_{i=1}^j V_i]$ contains no vertex in $V_j$, then 
\[\alpha(G[\cup_{i=1}^j V_i])=\alpha(G[\cup_{i=1}^{j-1} V_i])\le \alpha(G[(\cup_{i=1}^{j-1} V_i)\cup\{v\}]);\]
Else, it contains exactly one vertex $v\in V_j$, then $\alpha(G[\cup_{i=1}^j V_i])= \alpha(G[(\cup_{i=1}^{j-1} V_i)\cup\{v\}])$. Hence, 
$\alpha(G[\cup_{i=1}^j V_i])\le \max_{v\in V_j}\{\alpha(G[(\cup_{i=1}^{j-1} V_i)\cup\{v\}])\}$.
\end{proof}

\begin{proposition}\label{aa1} 
For $v\in V_j$,
\[
\alpha(G[(\cup_{i=1}^{j-1} V_i)\cup\{v\}])=
\max_{\substack{ v'\in V_{j-1}\cap N(v)\\v''\in V_{j-1}- N(v)}}
\{\alpha(G[(\cup_{i=1}^{j-2} V_i)\cup\{v'\}]),\alpha(G[(\cup_{i=1}^{j-2} V_i)\cup\{v''\}])+1\}
.\]
\end{proposition}

\begin{proof} 
Define $LHS:=\alpha(G[(\cup_{i=1}^{j-1} V_i)\cup\{v\}])$ and 
\[RHS:=\max_{\substack{ v'\in V_{j-1}\cap N(v)\\v''\in V_{j-1}- N(v)}}
\{\alpha(G[(\cup_{i=1}^{j-2} V_i)\cup\{v'\}]),\alpha(G[(\cup_{i=1}^{j-2} V_i)\cup\{v''\}])+1\}
.\]
Observe that the graph $G[(\cup_{i=1}^{j-2} V_i)\cup\{v'\}]$ is a subgraph of $G[(\cup_{i=1}^{j-1} V_i)\cup\{v\}]$ 
for any $v\in V_{j-1}$.
Also, the union of an independent set of $G[(\cup_{i=1}^{j-2} V_i)\cup\{v''\}]$ and $\{v\}$ is 
an independent set of $G[(\cup_{i=1}^{j-1} V_i)\cup\{v\}]$ for any $v''\in V_{j-1}- N(v)$.
The above two facts suffices to show that $RHS\le LHS $.
It remains to show $LHS\le RHS$. Note that 
$LHS$ can only be equal to either $\alpha(G[(\cup_{i=1}^{j-1} V_i)])$ or $\alpha(G[(\cup_{i=1}^{j-1} V_i)])+1$.

Assume that  $LHS= \alpha(G[(\cup_{i=1}^{j-1} V_i)])$.
Let $I$ be any maximum independent set of $G[\cup_{i=1}^{j-1} V_i]$. 
Then $I$ must contain some vertex $v'\in V_{j-1} \cap N(v)$, or we can add $v$ to 
$I$ to get an independent set of $G[(\cup_{i=1}^{j-1} V_i)\cup\{v\}]$ with size 
one greater than $\alpha(G[(\cup_{i=1}^{j-1} V_i)])$.
This contradicts our assumption.
Hence, $I$ is an independent set of $G[(\cup_{i=1}^{j-2} V_i)\cup\{v'\}]$ for 
some  $v'\in V_{j-1}\cap N(v)$. Thus, 
$LHS=|I|\le RHS$.

Now assume that  $LHS= \alpha(G[(\cup_{i=1}^{j-1} V_i)])+1$.
Let $I'$ be any maximum independent set of $G[(\cup_{i=1}^{j-1} V_i)\cup\{v\}]$. 
Then the size condition gives that $I'$ must contain $v$, 
and the set $J=I'-\{v\}$ is a maximum independent set of $G[(\cup_{i=1}^{j-1} V_i)]$.
If $J\cap V_{j-1}\neq\varnothing$, then it contains exactly one element which 
belongs to $V_{j-1}-N(v)$, say $v''$. Else, $J\cap V_{j-1}=\varnothing$, 
and it is contained in $\cup_{i=1}^{j-2} V_i$. 
Both cases imply that $J$ is an independent set of 
$G[(\cup_{i=1}^{j-2} V_i)\cap \{v''\}]$ for some $v''\in V_{j-1}-N(v)$.  
Thus, $|J|+1\le \alpha(G[(\cup_{i=1}^{j-2} V_i)\cap \{v''\}])+1$ for some $v''\in V_{j-1}-N(v)$.
Consequently, $LHS=|I'|\le RHS$. \end{proof}

The idea of finding $\alpha(G_P)$ in polynomial time is the following:
For each vertex $v$ in $V_j$ of $V(G_P)$, we iteratively compute 
 $\alpha(G[(\cup_{i=1}^{j-1} V_i)\cup\{v\}])$ for $j=1,\ldots,h-1$.
Initially, $\alpha(G[\{v\}])=1$ for any $v\in V_1$. 
For $v\in V_j$ and $j\ge 2$, we compute $\alpha(G[(\cup_{i=1}^{j-1} V_i)\cup\{v\}])$ using Proposition~\ref{aa1}.
To determine $\alpha(G[(\cup_{i=1}^{j-1} V_i)\cup\{v\}])$, we need to compare at most $|V_{j-1}|$ values. 
This can be done in no more than $|V(G)|$ steps. Therefore, determining all 
$\alpha(G[(\cup_{i=1}^{j-1} V_i)\cup\{v\}])$ for all $v\in V(G_P)$ takes at most $|V(G_P)|^2$ steps. 
After we acquire $\alpha(G[(\cup_{i=1}^{j-1} V_i)\cup\{v\}])$ for every $v\in V_j$, $j=1,\ldots,h-1$,
we compare $\alpha(G[(\cup_{i=1}^{h-2} V_i)\cup\{v\}])$ for $v\in V_{h-1}$ to get $\alpha(G_P)$. 
The whole process can be done in $O(|V(G_P)|^2)$ steps.

\subsection{The posets with $e(P)=\pi(P)$}

The results in~\cite{BurNag} contain not only the bound $\La(C_2,P)\le |P|+h-2$,
 but also the posets in~Figure~\ref{eep}, for which the equality $e(P)=\frac{1}{2}(|P|+h-2)$ holds. 
Furthermore, let $P_1$ and $P_2$ be any two posets satisfying the equality. 
Then the poset $P_1\oplus P_2$, obtained by $P_1$ and $P_2$ adding the relations $a < b$ for all $a\in P_1$ and $ b \in P_2$, and the poset $P_1\otimes P_2$, 
obtained by identifying the greatest element of $P_1$ to the least element of $P_2$ 
whenever $P_1$ has a greatest element and $P_2$ has a least element, 
also satisfy the equality (see~\cite{BurNag}, Lemma 4.4).
By Theorem~\ref{BurNag}, all posets above satisfy Conjecture~\ref{Conj:GriLu}.
\begin{figure}[ht]
\begin{center}

$V$
\begin{picture}(15,60)
\put(5,20){\circle*{4}}
\end{picture}
$B$
\begin{picture}(40,60) 
\put(5,10){\circle*{4}}
\put(25,10){\circle*{4}}
\put(25,30){\circle*{4}}
\put(5,30){\circle*{4}}
 \put(25,10){\line(-1,1){20}}
\put(25,10){\line(0,1){20}}
\put(5,10){\line(0,1){20}}
\put(5,10){\line(1,1){20}}
\end{picture}
$D_3$
\begin{picture}(60,60)
\put(5,30){\circle*{4}}
\put(25,10){\circle*{4}}
\put(25,30){\circle*{4}}
\put(25,50){\circle*{4}}
\put(45,30){\circle*{4}}
\put(25,10){\line(0,1){40}}
\put(25,10){\line(1,1){20}}
\put(25,10){\line(-1,1){20}}
\put(25,50){\line(1,-1){20}}
\put(25,50){\line(-1,-1){20}}
\end{picture}
$Q$
\begin{picture}(60,60)
\put(5,30){\circle*{4}}
\put(15,10){\circle*{4}}
\put(35,10){\circle*{4}}
\put(25,30){\circle*{4}}
\put(15,50){\circle*{4}}
\put(35,50){\circle*{4}}
\put(45,30){\circle*{4}}
\put(15,10){\line(1,2){10}}
\put(15,10){\line(-1,2){10}}
\put(35,10){\line(1,2){10}}
\put(35,10){\line(-1,2){10}}
\put(15,50){\line(1,-2){10}}
\put(15,50){\line(-1,-2){10}}
\put(35,50){\line(1,-2){10}}
\put(35,50){\line(-1,-2){10}}
\put(15,10){\line(3,2){30}}
\put(35,10){\line(-3,2){30}}
\put(15,50){\line(3,-2){30}}
\put(35,50){\line(-3,-2){30}}
\end{picture}
$S$
\begin{picture}(80,60)
\put(5,30){\circle*{4}}
\put(35,10){\circle*{4}}
\put(25,30){\circle*{4}}
\put(25,50){\circle*{4}}
\put(45,50){\circle*{4}}
\put(45,30){\circle*{4}}
\put(65,30){\circle*{4}}
\put(35,10){\line(1,2){10}}
\put(35,10){\line(-1,2){10}}
\put(35,10){\line(3,2){30}}
\put(35,10){\line(-3,2){30}}
\put(25,50){\line(0,-1){20}}
\put(25,50){\line(-1,-1){20}}
\put(25,50){\line(1,-1){20}}
\put(25,50){\line(2,-1){40}}
\put(45,50){\line(0,-1){20}}
\put(45,50){\line(1,-1){20}}
\put(45,50){\line(-1,-1){20}}
\put(45,50){\line(-2,-1){40}}
\end{picture}
$R$
\begin{picture}(80,75)
\put(5,30){\circle*{4}}
\put(35,10){\circle*{4}}
\put(25,30){\circle*{4}}
\put(25,50){\circle*{4}}
\put(45,50){\circle*{4}}
\put(45,30){\circle*{4}}
\put(65,30){\circle*{4}}
\put(35,70){\circle*{4}}
\put(5,50){\circle*{4}}
\put(65,50){\circle*{4}}
\put(35,10){\line(1,2){10}}
\put(35,10){\line(-1,2){10}}
\put(35,10){\line(3,2){30}}
\put(35,10){\line(-3,2){30}}
\put(25,50){\line(0,-1){20}}
\put(25,50){\line(-1,-1){20}}
\put(25,50){\line(1,-1){20}}
\put(25,50){\line(2,-1){40}}
\put(45,50){\line(0,-1){20}}
\put(45,50){\line(1,-1){20}}
\put(45,50){\line(-1,-1){20}}
\put(45,50){\line(-2,-1){40}}
\put(35,70){\line(1,-2){10}}
\put(35,70){\line(-1,-2){10}}
\put(35,70){\line(3,-2){30}}
\put(35,70){\line(-3,-2){30}}
\put(5,50){\line(0,-1){20}}
\put(5,50){\line(1,-1){20}}
\put(5,50){\line(3,-1){60}}
\put(5,50){\line(2,-1){40}}
\put(65,50){\line(0,-1){20}}
\put(65,50){\line(-1,-1){20}}
\put(65,50){\line(-3,-1){60}}
\put(65,50){\line(-2,-1){40}}
\end{picture}
\end{center}
\caption{Posets with $e(P)=\frac{1}{2}(|P|+h-2)$}\label{eep}
\end{figure}
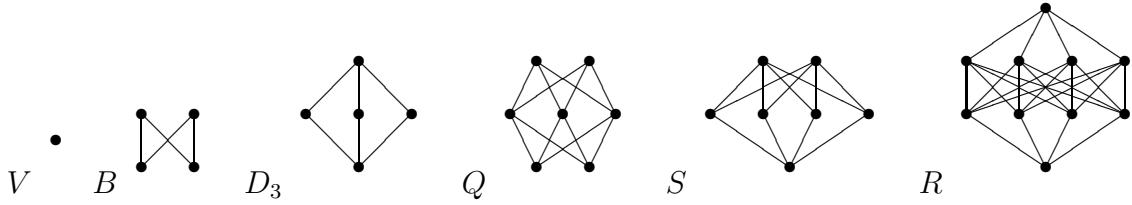
Imitating the previous ideas, 
we seek for posets which hold the equality 
\begin{equation}\label{ep}
e(P)=\frac{1}{2}(|P|+h-\alpha(G_P)-2). 
\end{equation}
This kind of posets satisfy Conjecture~\ref{Conj:GriLu} by Corollary~\ref{maincor}.
Observe that all posets in Figure~\ref{eep}, 
and the posets obtained by applying the operations $\otimes$ or $\oplus$ to them have $\alpha(G_P)=0$. 
So, in fact, these posets satisfy Equality~(\ref{ep}).
Does there exist other poset also satisfying Equality~(\ref{ep})?

Although we can compute $\alpha(G_P)$ in polynomial time, it was proven that 
computing $e(P)$ is NP-compete~\cite{Pal}. 
Thus, finding the target posets by computing these parameters is impractical. 
Instead, we will develop a method to construct posets that satisfy Equality~(\ref{ep}).  
Fix a graded poset $P$. 
We say that we perform a {\em $\Lambda$-extention to $P$ at level $i$}, 
if we add a new element $x$ to $P$, and  
the new relations that $x$ is less than any element in $L_{i+1}$  
and $x$ is greater than all but two elements in $L_{i-1}$ to get a new poset $P'$.
We only perform a $\Lambda$-extention to $P$ at level $i$ when $|L_{i-1}|\ge 3$. 
Note that the resulting poset $P'$ is not unique.
Analogously, we can perform a {\em $V$-extention on $P$ at level $i$} by adding a new element $x$ to $P$, 
and the new relations that $x$ is greater than any element in $L_{i-1}$  
and $x$ is less than all but two elements in $L_{i+1}$ to get a new poset.
Also, performing a $V$-extention to $P$ at level $i$ requires $|L_{i+1}|\ge 3$.
Each operation defined above gives a new vertex to the auxiliary graph.
So the independence number of the new graph $G_{P'}$ might increase by one.
Let us perform a $\Lambda$-extention or a $V$-extention to a poset $P$ with $\alpha(G_P)=0$, 
constructed by Burcsi and Nagy, to get a new poset $P'$. 
The next argument suffices to show that $P'$ satisfies Equality~(\ref{ep}):
\begin{eqnarray*}
\frac{1}{2}(|P|+h-2) 
&=&e(P)\le e(P')\\
&\le& \frac{1}{2}(|P'|+h-\alpha(G_{P'})-2)\\
&=&\frac{1}{2}(|P|+1+h-1-2)\\
&=&\frac{1}{2}(|P|+h-2).
\end{eqnarray*}
If $P''$ is a supboset of $P'$, and contains $P$ as a subposet, 
then $P''$ is graded and also satisfies Equality~(\ref{ep}) using above argument.
So, all the posets $P$ $P'$ and $P''$ satisfy Conjecture~\ref{Conj:GriLu}.

Performing the extentions many times to a poset is allowed.
Whenever the extention we perform to the poset every time can enlarge the independence number of the auxiliary graph, we obtain more posets satisfying Conjecture~\ref{Conj:GriLu}.
For example, the middle poset $S'$ in Figure~\ref{eepa} can be obitaned by performing a $\Lambda$-extension at level 3 and a $V$-extension at level 1 to $S$. 
The auxiliary graph of this poset is $2K_1$. Hence $\alpha(G_P)=2$.
The right poset $S''$ is a subposet of $S'$ and contains $S$ as a suboset.
Its auxiliary graph is the graph obtained by removing two incident edges from $K_4$, 
which has $\alpha(G_P)=2$.
Our theorem shows that both $S'$ and $S''$ satisfy Conjecture~\ref{Conj:GriLu}! 

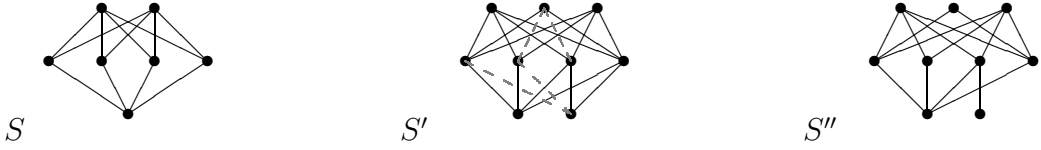
\begin{figure}[ht]
\begin{center}
 
$S$
\begin{picture}(80,60)
\put(5,30){\circle*{4}}
\put(35,10){\circle*{4}}
\put(25,30){\circle*{4}}
\put(25,50){\circle*{4}}
\put(45,50){\circle*{4}}
\put(45,30){\circle*{4}}
\put(65,30){\circle*{4}}
\put(35,10){\line(1,2){10}}
\put(35,10){\line(-1,2){10}}
\put(35,10){\line(3,2){30}}
\put(35,10){\line(-3,2){30}}
\put(25,50){\line(0,-1){20}}
\put(25,50){\line(-1,-1){20}}
\put(25,50){\line(1,-1){20}}
\put(25,50){\line(2,-1){40}}
\put(45,50){\line(0,-1){20}}
\put(45,50){\line(1,-1){20}}
\put(45,50){\line(-1,-1){20}}
\put(45,50){\line(-2,-1){40}}
\end{picture}
\hspace{50pt} 
$S'$
\begin{picture}(80,60)
\put(30,10){\circle*{4}}
\put(50,10){\circle*{4}}
\put(10,30){\circle*{4}}
\put(30,30){\circle*{4}}
\put(50,30){\circle*{4}}
\put(70,30){\circle*{4}}
\put(20,50){\circle*{4}}
\put(40,50){\circle*{4}}
\put(60,50){\circle*{4}}

\put(40,50){\line(3,-2){30}}
\put(40,50){\line(-3,-2){30}}

\put(20,50){\line(5,-2){50}}
\put(20,50){\line(3,-2){30}}
\put(20,50){\line(1,-2){10}}
\put(20,50){\line(-1,-2){10}}

\put(60,50){\line(-5,-2){50}}
\put(60,50){\line(-3,-2){30}}
\put(60,50){\line(1,-2){10}}
\put(60,50){\line(-1,-2){10}}

\put(30,10){\line(0,1){20}}
\put(30,10){\line(-1,1){20}}
\put(30,10){\line(1,1){20}}
\put(30,10){\line(2,1){40}}
\put(50,10){\line(0,1){20}}
\put(50,10){\line(1,1){20}}

\textcolor{gray}
{
\dashline{4}(50,10)(30,30)
\dashline{4}(50,10)(10,30)
\dashline{4}(40,50)(30,30)
\dashline{4}(40,50)(50,30)
}
\end{picture}
\hspace{50pt} 
$S''$
\begin{picture}(80,60)
\put(30,10){\circle*{4}}
\put(50,10){\circle*{4}}
\put(10,30){\circle*{4}}
\put(30,30){\circle*{4}}
\put(50,30){\circle*{4}}
\put(70,30){\circle*{4}}
\put(20,50){\circle*{4}}
\put(40,50){\circle*{4}}
\put(60,50){\circle*{4}}

\put(40,50){\line(3,-2){30}}
\put(40,50){\line(-3,-2){30}}

\put(20,50){\line(5,-2){50}}
\put(20,50){\line(3,-2){30}}
\put(20,50){\line(1,-2){10}}
\put(20,50){\line(-1,-2){10}}

\put(60,50){\line(-5,-2){50}}
\put(60,50){\line(-3,-2){30}}
\put(60,50){\line(1,-2){10}}
\put(60,50){\line(-1,-2){10}}

\put(30,10){\line(0,1){20}}
\put(30,10){\line(-1,1){20}}
\put(30,10){\line(1,1){20}}
\put(30,10){\line(2,1){40}}
\put(50,10){\line(0,1){20}}

\end{picture}

\end{center}
\caption{Posets with $e(P)=\frac{1}{2}(|P|+h-2-\alpha(G_{P}))$}\label{eepa}
\end{figure}

\Remark Although a poset $P$ can always be partitioned into $h$ antichains by Mirsky's theorem, 
the partition is not unique if the poset $P$ is not graded. For example, let $P$ be the poset consisting 
of elements $u_i$, $1\le i\le 4$, and $v_j$, $1\le j\le 12$ with the partial order relations 
$u_i\le u_{i'}$ if $i\le i'$, and $v_j\le u_4$ for $1\le j\le 12$. 
We can partition $P$ into $A_1=\{u_1\}$, $A_2=\{u_2\}$, $A_3=\{u_3,v_1,v_2,\ldots,v_{12}\}$ and $A_4=\{u_4\}$, or $A'_1=\{u_1,v_1,v_2,v_3,v_4\}$, $A'_2=\{u_2,v_5,v_6,v_7,v_8\}$, $A'_3=\{u_3,v_9,v_{10},v_{11},v_{12}\}$, and $A'_4=\{u_4\}$. 
In the first partition the auxiliary graph $G_P$ is the complete graph $K_{66}$, 
while in the second partition it will be a graph on $24$ vertices. 
To prevent the confusion, we only studied the graded posets.

\end{document}